\documentclass[12pt, reqno]{amsart}
\usepackage{amsfonts}
\usepackage{bbm}
\usepackage{amscd,amsfonts}
\usepackage{amssymb, eucal, amsfonts, amsmath, xypic, latexsym}
\usepackage{pifont}
\usepackage{mathrsfs,color}
\usepackage{amsthm,indentfirst,bm,fancyhdr,dsfont}
\usepackage{graphicx}
\usepackage[all]{xy}
\usepackage[CJKbookmarks=true]{hyperref}

\usepackage{mathrsfs}
\usepackage{amsmath}
\usepackage{amssymb}
\usepackage{hyperref}

\newtheorem{theorem}{Theorem}[section]
\newtheorem{lemma}[theorem]{Lemma}
\newtheorem{prop}[theorem]{Proposition}

\newtheorem{corollary}[theorem]{Corollary}
\theoremstyle{definition}

\newtheorem{remark}[theorem]{Remark}

\numberwithin{equation}{section}

\def\ggg{\mathfrak{g}}
\def\v{\mathfrak{v}}

\def\calv{\mathcal{V}}
\def\cm{\mathcal{M}}
\def\cz{\mathcal{Z}}
\def\cq{\mathcal{Q}}

\def\cs{{\mathcal{S}}}

\def\co{\mathcal{O}}
\def\cC{\mathcal{C}}
\def\cw{\mathcal{W}}
\def\calm{\mathcal{M}}
\def\cU{\mathcal{U}}

\def\sfn{\mathsf{N}}

\def\g{\mathfrak{g}}

\def\ggg{\mathfrak{g}}
\def\mmm{\mathfrak{m}}
\def\ppp{\mathfrak{p}}

\def\hhh{\mathfrak{h}}

\def\fh{\mathfrak{H}}

\def\bfk{\mathbf{k}}

\def\bba{\mathbb{A}}
\def\bbc{\mathbb{C}}
\def\bbf{\mathbb{F}}
\def\bbz{\mathbb{Z}}

\def\BbbB{\mathbf{B}}

\def\bbk{{\mathds{k}}}

\def\sfd{\textsf{d}}

\def\sfh{\textsf{h}}
\def\sft{\textsf{t}}
\def\tsN{\textsf{N}}

\def\Lie{\text{Lie}}
\def\ad{\text{ad}}

\def\det{\text{det}}

\def\Ad{\text{Ad}}

\def\gr{\text{gr}}

\def\Frac{\mbox{Frac}}
\def\rank{\text{rank}}

\def\Specm{\text{Specm}}
\def\reg{\text{reg}}

\def\rss{\text{rss}}

\def\ss{\text{ss}}

\def\pr1{\text{pr}_1}
\def\reg{\text{reg}}

\def\scrz{\mathscr{Z}}

\begin{document}

\title[Zassenhaus varieties]{On the Zassenhaus varieties of finite $W$-algebras  in prime characteristic} \author{Bin Shu and Yang Zeng}

\address{School of Mathematical Sciences, Ministry of Education Key Laboratory of Mathematics and Engineering Applications \& Shanghai Key Laboratory of PMMP,  East China Normal University, No. 500 Dongchuan Rd., Shanghai 200241, China} \email{bshu@math.ecnu.edu.cn}

\address{{School of Mathematics,} Nanjing Audit University, No. 86 West Yushan Rd., Nanjing, Jiangsu Province 211815, China}
\email{zengyang@nau.edu.cn}

  \thanks{This work is partially supported by the National Natural Science Foundation of China ({Nos. 12071136, 12271345,
12461005}), and by Science and Technology Commission of Shanghai Municipality (No. 22DZ2229014).}

\begin{abstract} Let $Z({\cw})$ be the center of the finite $W$-algebra $\cw({\ggg},e)$ associated with $\ggg=\Lie(G)$ and a nilpotent element $e\in\ggg$ for a  connected reductive algebraic group $G$ over an algebraically closed field $\bbk$ of prime characteristic $p$ under the standard hypotheses (H1)-(H3) (see \cite[\S6.3]{Jan2}).
In this paper, we first demonstrate  that our previous results in \cite{SZ} on the structure and geometric properties  of $Z({\cw})$ for $p\gg0$  are still true under the present weakened restriction on $p$.   Then we study the Zassenhaus variety $\scrz$ of $\cw({\ggg},e)$, which is by definition  the maximal spectrum $\text{Specm}(Z({\cw}))$ of $Z({\cw})$.  On basis of the structure properties of $Z({\cw})$, we describe $\scrz$ via a good transverse slice $\cs$ and show {{that $\scrz$ is birationally equivalent to $\cs$, thereby  a rational affine scheme. In the special case when $e=0$,  we reobtain  one of  the main results of \cite{T} on the rationality of the Zassenhaus varieites for reductive Lie algebras in prime characteristic.}}
\end{abstract}
\maketitle
\setcounter{tocdepth}{1}
\tableofcontents
\section{Introduction and Preliminaries}
\subsection{}
The study of finite $W$-algebras associated with a reductive Lie algebra and its regular nilpotent element over the complex number field $\bbc$  can be traced back to Kostant's work  (see \cite{Ko}). Kostant's construction was generalized to arbitrary even nilpotent elements by Lynch \cite{Lyn}.  Premet developed the finite $W$-algebras in full generality in \cite{Pre1}. Recall that  a finite $W$-algebra $U(\ggg_\bbc,\hat e)$ is a certain associative algebra associated to a finite-dimensional complex semisimple or reductive Lie algebra ${\ggg_\bbc}$ and a nilpotent  element $\hat e\in{\ggg}_\bbc$. We refer the readers to \cite{Pre1}, or to some other  survey papers (for example, \cite{A, L4, W}) for details.

As a counterpart of the finite $W$-algebra $U(\ggg_\bbc,\hat e)$ over $\bbc$, Premet defined  the finite $W$-algebra $\cw({\ggg},e)$ over a field $\bbk$ of characteristic  $p\gg0$  (see \cite{Pre3}).

\subsection{}
In Premet's formulation, the finite $W$-algebra  $\cw({\ggg},e)$ over ${\bbk}$
is obtained from the $\bbc$-algebra $U(\ggg_\bbc,\hat e)$ by the method of ``reduction modulo $p$". Hence the characteristic of field ${\bbk}$ must {be}
sufficiently large (see \cite{Pre3} for more details). In \cite{GT}, Goodwin and Topley generalized Premet's work to the case with $p=\text{char}(\bbk)$ satisfying Jantzen's standard hypotheses on the corresponding reductive algebraic group $G$ of $\ggg$ (cf. \cite[\S6]{Jan2}).

Along  Premet's formulation of finite $W$-algebras,  we developed in \cite{SZ} the theory of centers and the associated Azumaya property for the finite $W$-algebra $\cw({\ggg},e)$ in positive characteristic $p\gg0$. Our approach is roughly generalizing the classical theory of the centers of the universal enveloping algebra $U(\ggg)$ of the Lie algebra $\ggg$ of a semisimple algebraic group $G$, which was first developed  by Veldkamp (see \cite{Ve}, \cite{KW}, \cite{MiRu} and \cite{BGo}); and also applying Brown-Goodearl's arguments in \cite{BGl} on the Azumaya property for the algebras satisfying Auslander-regular and Maculay conditions to the modular finite $W$-algebra case.

\subsection{} One purpose of the present paper is to weaken the condition mentioned above for the main results in \cite{SZ}. Recall that there is another formulation  (\ref{equiW})  of finite $W$-algebras under the standard hypotheses (H1)-(H3) (\S\ref{sec: good grading}),   raised by Goodwin and Topley \cite{GT}, which is defined to be the invariants of Gelfand-Graev-Premet module $Q_\chi$  (\ref{eq: gelf gra mod}) under the action of the unipotent group $\mathcal{M}$ arising from the admissible subalgebra $\mmm$ (\ref{eq: adm sub}) introduced by Premet over $\bbk$.  This formulation is equivalent to the former one formulated by Premet whenever the former one can be defined (see \cite[Lemma 4.4 and Theorem 7.3]{GT}).  However, Goodwin-Topley's definition makes sense directly over $\bbk$, instead of ``reduction modulo $p$" through {the} $\bbc$-algebra $U(\ggg_\bbc,\hat e)$.  Making use of Goodwin-Topley's formation,  we will revisit the main results in \cite{SZ}, releasing the restriction $p\gg0$ there.

\subsubsection{}\label{sec: good grading} Let us  recall some basic materials.
Let $G$ be a connected reductive algebraic group over an algebraically closed field $\bbk$ of positive  characteristic $p$, and ${\ggg}=\text{Lie}(G)$. We introduce the following standard hypotheses on $G$ and $p$:
\begin{enumerate}
  \item [(H1)] The derived group $\mathcal DG$ of $G$ is { simply connected};
  \item [(H2)] The prime $p$ is good for $\ggg$;
  \item [(H3)] There exists a $G$-invariant non-degenerate bilinear form on $\ggg$.
\end{enumerate}
As to the explanation on (H1)-(H3), the reader refers to \cite[\S6.3-6.4]{Jan2}.

{
Let $g \in G$ and $x \in \ggg$.  We write $g.x$ for the image of $x$ under $g$ in the adjoint action,
$G_x$ for the centralizer of $x$ in $G$ and $\ggg_x$ for the centralizer of $x$ in $\ggg$.  Also we
define $G_{[x]} := \{g \in G \mid g.x \in \bbk x\}$ and call it the {\it normalizer of $x$ in $G$}.
We use similar notation when considering the coadjoint action of $G$ on $\ggg^*$.

For a given nilpotent element $e \in \ggg$, let $T$ be a maximal torus in $G$ containing a maximal torus $T_{[e]}^\circ$ of $G_{[e]}$,
and let $\Phi$ be the root system of $G$ with respect to $T$. Write $X_*(T)$ for the group of cocharacters of $T$. We write $G^\circ$ for the identity component of $G$, and define $T^{e}:=(T_{e} \cap \mathcal{D} G)^\circ$, $T^{[e]}:=(T_{[e]} \cap \mathcal{D} G)^\circ$.
By \cite[Lemma 5.3]{Jan3} and our choices, there exists a cocharacter
$\lambda: \bbk^\times \to G$ associated to $e$ with $\lambda\in X_*(T^{[e]})$, which means that $\lambda(t).e = t^2 e$ and $\lambda(\bbk^\times)$ is contained in the derived subgroup of
a Levi subgroup in which $e$ is distinguished (see \cite[Definition 5.3]{Jan3}).}

In \cite{Po},  Pommerening showed that under { the hypothesis (H2) and $p>2$ (note that $\text{GL}_n$ satisfies (H1)-(H3) for all primes)}, for any given nilpotent element  $e\in\ggg$, there is an $\mathfrak{sl}_2$-triple $(e, h, f)$ in $\ggg$,\footnote{Pommerening's result greatly improves the restriction on $p$ in the reference  \cite[\S5.5]{Cart}. In \cite{ST}, Stewart and Thomas further proved that under the assumption $p>2$, {there} is an bijection between  the set of conjugacy classes of $\mathfrak{sl}_2$-triples  and  the set of  nilpotent orbits in $\ggg$ if and only if the odd prime $p>\sfh$ where $\sfh$ is the Coxeter number of $G$. In \cite[Theorem 1.7]{ST}, the existence of $\mathfrak{sl}_2$-triple for a nilpotent element $e$ can be established in the case when { $p\geqslant3$} with some exceptional case only involving $G_2$.}
and { the associated cocharacter $\lambda$ defines a ${\bbz}$-grading $\ggg=\bigoplus_{i\in{\bbz}}{\ggg}(i)$ with $\ggg(i) := \{x \in \ggg \mid \lambda(t).x = t^i x \text{ for all } t \in \bbk^\times\}$.} Such a grading will be called {\it the Dynkin grading}. 

Since $\ggg$ is the Lie algebra of an algebraic group, it is a restricted Lie algebra in a natural way, and we write $x \mapsto x^{[p]}$ for the
$p$th power map.
Under the hypothesis (H3), we have   the $G$-equivariant non-degenerate bilinear form $(\cdot,\cdot)$ on $\ggg$ such that $(e, f)=1$, and define $\chi\in\ggg^*$ by letting $\chi(x)=(e, x)$ for all $ x\in{\ggg}$. Furthermore, we can define an isomorphism $\kappa$ from $\ggg$ to $\ggg^*$ via $(\cdot,\cdot)$.
{ Let $\mathfrak{l}$ be a Lagrangian subspace of $\ggg(-1)$ and $\mathfrak{l}'$ the annihilator of $\mathfrak{l}$ with respect to $(e,[\cdot,\cdot])$ such that $\ggg(-1)=\mathfrak{l}\oplus\mathfrak{l}'$. We further assume that the
isotropic subspaces 
$\mathfrak{l}$ and $\mathfrak{l}'$ are $T$-stable under the adjoint action and $\chi(\mathfrak{l}^{[p]}) = 0$, and
this second condition automatically holds if $p \neq 2$.}
Then we define an admissible nilpotent subalgebra
\begin{align}\label{eq: adm sub}
\mmm:=\mathfrak{l}'\oplus\bigoplus_{i\leqslant-2}\ggg(i)
 \end{align}
 of $\ggg$. { We note that $\mmm$ is a
restricted subalgebra of $\ggg$, and it is stable under the adjoint
action of $T$.}
There exists a subset $\Phi(\mmm) \subset \Phi$ with
$
\mmm=\bigoplus_{\nu\in \Phi(\mmm)}\ggg_\nu,
$ and $\ggg_\nu$ is the corresponding root space of $\ggg$.
Set ${\ppp}:=\bigoplus_{i\geqslant 0}{\ggg}(i)$ and $\tilde{\mathfrak{p}}:={\ppp}\oplus\mathfrak{l}$. Set
\begin{align}\label{eq: gelf gra mod}
Q_{\chi}:=U({\ggg})\otimes_{U(\mathfrak{m})}{\bbk}_\chi.
\end{align}
Let $I_{\chi}$ denote the left ideal of the universal enveloping algebra $U({\ggg})$ of $\ggg$ generated by all $x-\chi(x)$ with $x\in\mathfrak{m}$. Then we further have $Q_{\chi}\cong U(\ggg)/I_\chi$ as ${\ggg}$-modules via the ${\ggg}$-module map sending $1+I_\chi$ to $1_\chi$ where $1_\chi$ is a generator of $Q_\chi$ over $U(\ggg)$.

\subsubsection{}\label{sec: unip M} { The one-parameter subgroup $u_\nu:\bbk\rightarrow G$ associated with the roots $\nu\in \Phi(\mmm)$ generate
a connected unipotent subgroup  $\mathcal{M}$ of $G$ defined over $\mathbb F_p$, stable under conjugation by $T$ and such that $\mmm= \Lie(\mathcal{M})$.}
     It follows from \cite[Lemma 4.1]{GT} that $I_\chi$ is stable under the adjoint action of $\mathcal{M}$. Then
the adjoint action of $\mathcal{M}$ on $Q_\chi$ given by $g.(u+I_\chi)=(g.u)+I_\chi$
for $g\in\mathcal{M}$ and $u\in U(\ggg)$ descends to an adjoint action on {the graded module $\text{gr}(Q_\chi)$ under the Kazhdan grading as defined in \S\ref{221}.}

{
We denote by $\Phi^e \subset X^*(T^e)$ the \emph{restricted root system}: the set of non-zero restrictions $\alpha|_{T^e}$ where
$\alpha \in \Phi$. For $\alpha \in X^*(T^e)$, we write $\ggg_\alpha:=\{x\in\ggg \mid t.x = \alpha(t)x\text{ for all }t \in T^e\}$
for the $T^e$-weight space corresponding to $\alpha$. Incorporating the Dynkin grading, we obtain the decompositions
\begin{equation}\label{rest dec}
\ggg = \bigoplus_{j \in \mathbb{Z}} \ggg_0(j) \oplus \bigoplus_{\substack{\alpha \in \Phi^e \\ j \in \mathbb{Z}}} \ggg_\alpha(j),
\end{equation}
where $\ggg_\alpha(j):=\ggg_\alpha\cap\ggg(j)$.

The space $[\ggg, e]$ is stable under the action of $T^{[e]}$, so
decomposes in the way as \eqref{rest dec}.  We may
choose a $T^{[e]}$-stable complement $\v$ to $[\ggg,e]$ in $\ggg$. 
Since $\ggg_e$ is orthogonal to $[\ggg,e]$ with respect to $(\cdot,\cdot)$, we see that the map $\ggg_e  \to \v^*$ given by $x \mapsto (x,\cdot)$ is an isomorphism, or in other words that the form $(\cdot, \cdot)$ restricts to a non-degenerate pairing
\begin{equation*} \label{e:vdual}
\ggg_e \times \v \to \bbk,
\end{equation*} i.e., $\v$ is dual to $\ggg_e$ via $(\cdot,\cdot)$.}

Define the finite $W$-algebra associated with the pair $(\ggg,e)$ as the invariants of $Q_\chi$ under $\calm$-action, i.e.,
\begin{equation}\label{equiW}
\cw({\ggg},e):=Q_\chi^{\mathcal{M}}.
\end{equation}
This is a $\bbk$-algebra (see for example, \cite[Lemma 4.2]{GT}).
The finite $W$-algebra $\cw(\ggg,e)$ over $\bbk$ can be regarded as 
{ a filtered deformation of the coordinate algebra of a so-called
{\it good transverse slice} $e+\v$ to the nilpotent orbit of $e$ (see \cite[Theorem 5.2]{GT}) as was introduced by Spaltenstein
in \cite{Spa}.}


\subsubsection{}\label{sec: 0.3.3}
Denote by $Z(\ggg)$ the center of $U(\ggg)$.
Let $\varphi$ be the natural representation of $\ggg$ over $Q_\chi$, which induces a $\mathds{k}$-algebra homomorphism from $Z(\ggg)$ to the so-called extended finite $W$-algebra
$$U(\ggg,e):=(\text{End}_{{\ggg}}Q_{\chi})^{\text{op}}$$
 such that $(\varphi(x))(1_\chi)=x.1_\chi\in Q_\chi$ for any $x\in Z(\ggg)$.
  Let $Z_p(\ggg)$ be the $p$-center of $U(\ggg)$, which is generated by  elements $x^p-x^{[p]}$ with $x\in\ggg$.  
Then we have $\bbk$-algebra isomorphisms $\varphi(Z_p(\ggg))\cong\varphi(Z_p(\tilde{\mathfrak{p}}))\cong Z_p(\tilde{\mathfrak{p}})$ (see the proof of \cite[Theorem 2.1]{Pre3}). Now
we identify $Z_p(\tilde{\mathfrak{p}})$ with $\varphi(Z_p(\tilde{\mathfrak{p}}))$ and $\varphi(Z_p(\ggg))$.
Define its invariant subalgebra $Z_p(\tilde{\mathfrak{p}})^{\mathcal{M}}$ under the action of $\Ad\,\mathcal{M}$ as in \S\ref{sec: unip M} and in \S\ref{p-center} for more precise description.

  Denote by $r$ the rank of $\ggg$. Recall that under the hypotheses  (H1)-(H3), by Veldkamp's theorem \cite{Ve,MiRu,BGo} we know that $Z(\ggg)$ is generated by the $p$-center $Z_0(\ggg):=Z_p(\ggg)$ and the Harish-Chandra center $Z_1(\ggg):=U(\ggg)^G=\{u\in U(\ggg)\mid g.u=u~\text{for all}~\ggg\in G\}$
of $U(\ggg)$. Moreover, there exist algebraically independent generators $g_1,\ldots,g_r$ of $U(\ggg)^G$ such that $Z(\ggg)$ is a free $Z_p(\ggg)$-module of rank $p^r$ with a basis consisting of all $g_1^{t_1}\cdots g_r^{t_r}$ with $0\leqslant t_k\leqslant p-1$ for all $k$.

\subsubsection{} Let $Z(\cw)$ denote the center of  $\cw({\ggg},e)$.
Write $Z_0(\cw):=Z_p(\tilde{\mathfrak{p}})^{\mathcal{M}}$, which lies in $Z(\cw)$ (see \S\ref{sec: p center more}), called the $p$-center of $\cw({\ggg},e)$. Note that the image of $Z_1(\ggg)$ under the map $\varphi$ 
also lies in $Z(\cw)$, and we denote it by $Z_1(\cw)$. Set $f_i:=\varphi(g_i)$ with $1\leqslant i\leqslant r$, which are all in $Z_1(\cw)$.
%

\subsubsection{} Denote by $\Lambda_r:=\{(i_1,\ldots,i_r)\mid i_j\in\{0,1,\ldots,p-1\}\}$ with $1\leqslant j\leqslant r$.
In the present paper, we will show that the following main results in \cite{SZ} for $p\gg0$ are still true with the weakened condition on $p$.

\begin{theorem}
\label{central thm} Under the hypotheses (H1)-(H3) in \S\ref{sec: good grading}, 
the following statements  hold.
\begin{itemize}
\item[(1)] The $\mathds{k}$-algebra $Z(\cw)$ is generated by  $Z_0(\cw)$ and $Z_1(\cw)$. More precisely, $Z(\cw)$ is a free module of rank $p^r$ over $Z_0(\cw)$ with basis $f_1^{t_1}\cdots f_r^{t_r}$, where $(t_1,\ldots,t_r)$ runs through $\Lambda_r$.
\item[(2)] The multiplication map $\mu:~Z_0(\cw)\otimes_{Z_0(\cw)\cap Z_1(\cw)} Z_1(\cw)\rightarrow Z(\cw)$ is an isomorphism of
$\mathds{k}$-algebras.
\item[(3)] The locus of points in $\Specm(Z(\cw))$ that occur as $Z(\cw)$-annihilators of irreducible $\cw(\ggg,e)$-module of maximal dimension coincides with the smooth locus consisting of smooth  points in $\Specm(Z(\cw))$.
\end{itemize}
\end{theorem}

The proof of the above theorem is given in \S\ref{sec: proof 01}.

\subsection{Zassenhaus varieties $\Specm(Z({\cw}))$}
 As usual, we always denote by $\Specm(R)$ the maximal spectrum for a commutative $\bbk$-algebra $R$ (the spectrum of maximal ideals of $R$).

The maximal spectrum of the center subalgebra of finite $W$-algebra $\cw({\ggg},e)$ is called  the Zassenhaus variety  just in  the same way as in the case of modular Lie algebras.
\subsubsection{The description of the Zassenhaus varieties}

In the remaining part  we first obtain a precise description of  the Zassenhaus variety $\Specm(Z(\cw))$ for the finite {$W$-algebras.}

\begin{theorem}\label{thm: main 2} There is an isomorphism of schemes
  \begin{align}\label{eq: fiber product-1}
 \Specm(Z(\cw)){\overset\cong\longrightarrow} \kappa(\cs)^{(1)}
 \times_{\calv}\hhh^*\slash W,
  \end{align}
 where { $\cs=e+\v$ (see (\ref{eq: S})),} the superscript $(\cdot)^{(1)}$ denotes Frobenius twist (see \S\ref{sec: Frob twist} for the definition),
$\hhh$ is a given maximal toral subalgebra of $\ggg$ and $W$ denotes the Weyl group of $\ggg$, and $\calv$ is the affine variety defined by $Z_0({\cw})\cap Z_1({\cw})$ which turns out to be a polynomial ring of rank $r=\dim \hhh$.
\end{theorem}
The proof of this theorem is given in \S\ref{sec: main them 3}. It is worth mentioning that $\cs$ is isomorphic to $\v$ as an affine variety (\S\ref{sec: slodowy trans}).

\subsubsection{The rationality of $\Specm(Z(\cw))$} Recall that an irreducible affine algebraic variety $X$ over a field $\bfk$ is called rational if the fraction field $\text{Frac}(\bfk[X])$ is purely transcendental over $\bfk$. Tange in \cite{T} proved that the Zassenhaus variety $Z(\ggg)$  over $\bbk$ under the hypotheses  (H1)-(H3) is an affine rational variety. This confirmed a conjecture raised by J. Alev, and extended his previous work jointly with Premet (see \cite{PT}). Based on the above theorem on the structure of the Zassenhaus variety of $Z(\cw)$, we further verify its rationality. This is to say, we will prove the following result.

\begin{theorem}\label{thm: rationality} The affine variety $\Specm(Z(\cw))$ is a rational variety.
\end{theorem}
To verify the rationality of a certain affine variety it is a natural  way to show that this variety  is bi-rationally equivalent to some affine space. So we need to construct an open dense subset of $\Specm(Z(\cw))$ such that  it is isomorphic to an open subset of certain affine space. We roughly take the strategy in \cite{T}. The critical difference lies in
the fiber product decomposition  (\ref{eq: fiber product-1}) of $\Specm(Z(\cw))$ where  certain integrality like
  $$\Specm(Z(\ggg))=\ggg^{*(1)}\times_{\hhh^{*(1)}\slash W}\hhh^*\slash W$$
 in the case of reductive Lie algebras (happening on $\ggg^{*(1)}$) does not seem to be kept. The key point for recovery of such integrality happening on $\kappa(\cs)$ is the fact that $\cs$ contains a nonempty open subset consisting of regular semisimple elements (see Proposition \ref{prop: regular ss open den}).

  The final proof of Theorem \ref{thm: rationality} is given in \S\ref{sec: proof 02}.

{{
\subsection{} When $e=0$,  $\cw:=\cw(\ggg,e)$ is exactly the universal enveloping algebra $U(\ggg)$ of $\ggg$. The Zassenhaus variety $\Specm(Z(\cw))$  coincides with $\Specm(Z(\ggg))$. Hence Theorem \ref{thm: rationality} covers \cite[Theorem 1]{T}, the latter of which asserts the rationality of the Zassenhaus variety of a reductive Lie algebra over $\bbk$.

}}

\section{Finite $W$-algebras and their centers}\label{sec: 1}

In this section, we continue to recall some basic material on finite $W$-algebras and their centers over $\bbk$. Here and throughout $\bbk$ is an algebraically closed field of characteristic $p>0$.

\subsection{The Kazhdan grading of finite $W$-algebras}\label{221}


 As in \S\ref{sec: unip M}, { $\mathcal{M}$ denotes the connected unipotent subgroup of $G$ which is generated by one-parameter {subgroup} $u_\nu:\bbk\rightarrow G$ associated with the roots $\nu\in \Phi(\mmm)$.} 
     It follows from \cite[Lemma 4.1]{GT} that $I_\chi$ is stable under the adjoint action of $\mathcal{M}$. With the identification between $Q_\chi$ and $U(\ggg)\slash I_\chi$,
the adjoint action of $\mathcal{M}$ on $Q_\chi$ {is given by $g.(u+I_\chi)=(g.u)+I_\chi$
for $g\in\mathcal{M}$ and $u\in U(\ggg)$.}
Recall that  the finite $W$-algebra associated with the pair $(\ggg,e)$ is defined as $\cw(\ggg,e)=Q_\chi^{\mathcal{M}}$. So we have
\begin{equation*}\label{W2}
\cw(\ggg,e)=\{u+I_\chi\in Q_\chi\mid g.u+I_\chi=u+I_\chi~\text{for all}~g\in\mathcal{M}\}.
\end{equation*}
 On the other hand,
the adjoint action of $\mathcal{M}$ on $Q_\chi$ induces an action on $(\text{End}_{\ggg}Q_{\chi})^{\text{op}}$ {(the opposite
algebra of the algebra of endomorphisms of $Q_\chi$)}  by $(g.f)(u+I_\chi)=g.(f(g^{-1}.u+I_\chi))$ for $g\in\mathcal{M}$, $f\in\text{End}_{\ggg}Q_{\chi}$, $u\in U(\ggg)$, and the invariant {subalgebra
 is} denoted by $(\text{End}_{\ggg}^\mathcal{M}Q_{\chi})^{\text{op}}$.

By the same discussion as the case over $\bbc$, there exists an isomorphism
$$U({\ggg},e)=(\text{End}_{\ggg}Q_{\chi})^{\text{op}}\cong Q_{\chi}^{\text{ad}\,{\mmm}}$$
by sending $u+I_\chi$ to $u.1_\chi$ for any $u\in U({\ggg},e)$.
From now on, we will identify $(\text{End}_{{\ggg}}Q_{\chi})^{\text{op}}$ with $Q_{\chi}^{\text{ad}\,{\mmm}}$. By  definition, $\cw({\ggg},e)$  can be naturally identified with a subalgebra of the extended finite $W$-algebra $U({\ggg},e)$ over ${\bbk}$. { Denote by $n$ the dimension of $\ggg$. With respect to the Dynkin grading mentioned in \S\ref{sec: good grading}, let $Y_1,\ldots,Y_{n}$ be a homogeneous basis of $\ggg$ with $Y_i\in \ggg(\text{deg}Y_i)$ for $i=1, \ldots, n$.
The Kazhdan filtrations on $Q_\chi$ and also $\cw({\ggg},e)$ can  be defined, arising from the filtration $\{U_j(\ggg)\mid j\in\bbz\}$ of $U(\ggg)$ with $U_j(\ggg)$ spanned by the PBW basis $Y_{i_1}\cdots Y_{i_k}$ with $\sum_{t=1}^k{(\text{deg}Y_{i_t}+2)}{\leqslant} j$ for $1\leqslant i_1\leqslant \ldots\leqslant i_k\leqslant n$}.  The gradings associated with the Kazhdan filtrations are called {\it {the Kazhdan gradings}}.
The adjoint action of $\mathcal{M}$ on $Q_\chi$  
descends to the adjoint action on its graded counterpart $\text{gr}(Q_\chi)$. Write $\text{gr}(\cw({\ggg},e))$ for the gradation of the finite $W$-algebra $\cw({\ggg},e)$, which is a subspace of $\text{gr}(Q_\chi)$.

\subsection{The ring-theoretic property of finite $W$-algebras}\label{ring theory property}
To discuss the related topics on finite $W$-algebras and their subalgebras, we first need the following observation.
\begin{lemma} \label{Noe Pri} Under the hypotheses  (H1)-(H3), the following statements hold.
\begin{itemize}
\item[(1)] 
    The gradation $\text{gr}(\cw({\ggg},e))$ is a unique factorization domains; 
\item[(2)] The finite $W$-algebra $\cw({\ggg},e)$  is a  Noetherian ring;
\item[(3)] $\cw({\ggg},e)$ is a prime ring, which does not contain any zero-divisor.
\end{itemize}
\end{lemma}
{\begin{proof}
Under the standard hypotheses (H1)-(H3), 
it follows from \cite[Theorem 7.3 and (7.4)]{GT}
that there exists a $\mathds{k}$-algebras isomorphism $\bar\psi:\gr(\cw(\ggg,e))\xrightarrow\sim S(\ggg_e)$, which
generalizes the case of characteristic $p\gg0$ in \cite[(1.6)]{SZ}.
So the gradation of
$\cw({\ggg},e)$ is isomorphic to a polynomial algebra. Then the standard filtration
arguments work, as these properties hold for the associated graded algebra. More details are omitted here.
\end{proof}}

\subsection{The $p$-centers of finite $W$-algebras}\label{p-center}
\subsubsection{} Let $\varphi$ be the natural representation of $\ggg$ over $Q_\chi$, which induces a $\mathds{k}$-algebras homomorphism from $Z(\ggg)$ to $(\text{End}_{{\ggg}}Q_{\chi})^{\text{op}}$ such that
\[\begin{array}{lcll}
\varphi:&Z(\ggg)&\rightarrow&(\text{End}_{\ggg}Q_{\chi})^{\text{op}}\\ &x&\mapsto&l_x
\end{array}
\]where $l_x(1_\chi)=x.1_\chi\in Q_\chi$ for any $x\in Z(\ggg)$.  The map $\varphi$ will play a critical rule in the { subsequent} arguments.

Recall that the associated graded algebra of $Q_\chi$ under the Kazhdan grading is $\text{gr}(Q_\chi)=\text{gr}(S(\ggg))/\text{gr}(I_\chi)$, and by the PBW theorem we have that $S(\ggg)=S(\tilde{\mathfrak{p}})\oplus\text{gr}(I_\chi)$. As $\text{pr}: S(\ggg)\rightarrow S(\tilde{\mathfrak{p}})$ is the projection along this direct sum decomposition, this restricts to an isomorphism
$\text{gr}(Q_{\chi})\cong S(\tilde{\mathfrak{p}})$.

The adjoint action of $\mathcal{M}$ as defined in \S\ref{221} on $Q_\chi$ descends to an adjoint action on $\text{gr}(Q_\chi)$ and this gives a twisted action of $\mathcal{M}$ on $S(\tilde{\mathfrak{p}})$ defined by $\text{tw}(g).f:=\text{pr}(g.f)$
for $g\in\mathcal{M}$ and $f\in S(\tilde{\mathfrak{p}})$, where $g.f$ denotes the usual adjoint action of $g$ on $f$ in $S(\ggg)$. We write $S(\tilde{\mathfrak{p}})^{\mathcal{M}}$ for the invariants with respect to this action.

 \subsubsection{Frobenius twist}\label{sec: Frob twist}  We denote for a scheme $X$ over $\bbk$ by $X^{(1)}$ the Frobenius twist of $X$, which means  the same
 scheme twisted by the Frobenius morphism, i.e.,
the structure sheaf is exponentiated to the $p$th power (see \cite[\S{I}.9.1]{Jan1} for the details). This notation will be used through the paper.


{ Recall that there is a homomorphism $\mathfrak{f}: S(\ggg^{(1)})\rightarrow Z_p(\ggg)\subset U(\ggg)$ via sending $x\mapsto x^p-x^{[p]}$ for all $x\in\ggg$. This
$\mathfrak{f}$ is a $G$-equivariant isomorphism, which induces a $G$-equivariant isomorphism from $\text{Specm}(Z_p(\ggg))$ to $(\ggg^{(1)})^*$.  Moreover, $\mathfrak{f}|_{S({\hhh}^{(1)})}$ becomes  a homomorphism from $S({\hhh}^{(1)})$ to $U({\hhh})=S({\hhh})$. Thus, $\mathfrak{f}|_{S({\hhh}^{(1)})}$ gives rise to the induced morphism

$$\mathfrak{F}:{\hhh}^*\rightarrow ({\hhh}^{(1)})^*,$$
where the notation  $ ({\hhh}^{(1)})^*$ (resp. $ (\ggg^{(1)})^*$)  will be often simply denoted by ${\hhh}^{(1)*}$ (resp.  $\ggg^{(1)*}$).  We have that $\mathfrak{F}(\lambda)(h)=\lambda(h)^p-\lambda(h^{[p]})$ for all $\lambda\in \hhh^*$ and $h\in\hhh$.
}



Note that  $\ggg^{(1)*}$ is isomorphic to $\ggg^{*(1)}:=(\ggg^*)^{(1)}$.
In the subsequent, $\text{Specm}(Z_p(\ggg))$ is identified  with $\ggg^{(1)*}$ or $\ggg^{*(1)}$ according to the requirement of the context.

\subsubsection{Some {materials} on $Z_0(\cw)$}\label{sec: p center more}
We write $I_p:=\varsigma(\mathfrak{m}^{(1)})Z_p(\ggg)$ for $\varsigma(\mathfrak{m}^{(1)}):=\{x^p-x^{[p]}-\chi(x)^p\mid x\in\mathfrak{m}\}$.
Since the group $\mathcal{M}$ preserves the left ideal $I_\chi$ and $\varsigma$ is $G$-equivariant, hence $\mathcal{M}$ acts on both $U(\ggg,e)\cong Q_{\chi}^{\text{ad}\,{\mmm}}$ and $Z_p(\ggg)/I_p$.
 Let $U(\mathfrak{g},e)^{\mathcal{M}}$ and $(Z_p(\ggg)/I_p)^{\mathcal{M}}$  denote the fixed point subspace of $U(\mathfrak{g},e)$ and $Z_p(\ggg)/I_p$ under the action of $\text{Ad}\,\mathcal{M}$, respectively.

 On the other hand, since $Z_p(\tilde{\mathfrak{p}})\cong\varphi(Z_p(\tilde{\mathfrak{p}}))
 =\varphi(Z_p(\ggg))\cong Z_p(\ggg)/I_p$ by definition, then we have  $Z_p(\tilde{\mathfrak{p}})^{\mathcal{M}}
 \cong (Z_p(\ggg)/I_p)^{\mathcal{M}}$ as $\bbk$-algebras.
Hence, $Z_0(\cw)=Z_p(\tilde{\mathfrak{p}})^{\mathcal{M}}$ {indeed} lies in the center of the finite $W$-algebra $\cw({\ggg},e)$.

Moreover, as stated in \cite[Lemma 4.4]{GT} {(see also \cite[Remark 2.1]{Pre2})} we can identify $\cw(\ggg,e)$ with $U(\ggg,e)^\calm$, i.e., we have
\begin{align}\label{eq: def Z0}
&\cw(\ggg,e) =U(\mathfrak{g},e)^{\mathcal{M}}\;\text{ and }\cr
&\varphi(Z_p(\tilde{\mathfrak{p}}))^{\calm}\cong Z_p(\tilde{\mathfrak{p}})^{\mathcal{M}}=Z_p(\tilde{\mathfrak{p}})\cap \cw(\ggg,e).
\end{align}

Under the weakened restriction on $p$, the results {in} \cite[Proposition 2.1]{SZ} as below  still hold.
\begin{prop}\label{PI and Central S} 
Under the hypotheses  (H1)-(H3), we have
\begin{itemize}
\item[(1)] the ring $Z_0(\cw)$ is  Noetherian;
\item[(2)] the ring $Z(\cw)$ is integrally closed;
\item[(3)] the ring extension $Z(\cw)/Z_0(\cw)$ is integral, and $Z(\cw)$ is finitely-generated as a $Z_0(\cw)$-module. In particular, $Z(\cw)$ is a finitely-generated commutative algebra over $\bbk$;
\item[(4)]$\cw(\ggg,e)$ is a PI ring.
\end{itemize}
\end{prop}

\subsection{}\label{for fintecen}

Set $\mathfrak{a}:=\{x\in\tilde{\mathfrak{p}}\mid (x,\v)=0\}$.
The following lemma shows  the relationship between the $p$-centers of $\cw({\ggg},e)$ and $U(\ggg,e)$:
\begin{lemma}({\cite[(8.3)]{GT}})\label{pmge}
There exists an isomorphism between  $\mathds{k}$-algebras
 $$Z_p(\widetilde{\mathfrak{p}})^{\mathcal{M}}\otimes_{\mathds{k}}Z_p(\mathfrak{a})\cong Z_p(\tilde{\mathfrak{p}})$$
 under the multiplication map.
\end{lemma}

In \cite[Theorem 8.4]{GT},  Goodwin and Topley introduced the following transition property between finite $W$-algebras and their extended counterparts {(see \cite[Theorem 2.1]{Pre3} and \cite[Remark 2.1]{Pre2} for the case of characteristic $p\gg0$ by Premet).} 
\begin{theorem} (\cite{GT, Pre3, Pre2})\label{prem}
Under the hypotheses  (H1)-(H3), we have
\begin{itemize}
\item[(1)] The algebra ${U}(\ggg,e)$ is generated by its subalgebras $\cw({\ggg},e)$ and $Z_p(\tilde{\mathfrak{p}})$;
\item[(2)] ${U}(\ggg,e)\cong \cw({\ggg},e)\otimes_\bbk Z_p(\mathfrak{a})$ as $\mathds{k}$-algebras;
\item[(3)] $U(\ggg,e)$ is a free  module over $Z_p(\tilde{\mathfrak{p}})$ of rank $p^{\dim\ggg_e}$;
\item[(4)] $\cw({\ggg},e)$ is a free  module over $Z_p(\tilde{\mathfrak{p}})^{\mathcal{M}}$ of rank $p^{\dim  \ggg_e}$.
\end{itemize}
\end{theorem}

\section{{ Good transverse} slices}\label{sec: slo sli}
Maintain the notations and assumptions as before. In particular, we suppose the hypotheses (H1)-(H3) for $G$ unless other stated. Then these assumptions are naturally satisfied for any Levi subgroups.

\subsection{{ Good transverse} slices}\label{Slo}
    Now we  need to talk more related to { good transverse} slices, continuing the introduction in \S\ref{sec: unip M}  (see \cite{Spa}, or \cite{GT}, \cite[{\S7}]{Jan3},  \cite{Pre1} for further details). 

          {We can identify $S(\ggg)$ with the algebra $\mathds{k}[\ggg^*]$ of regular functions on the affine variety $\ggg^*$.
Let $\mathfrak{m}^\perp$ denote the set of all linear functions on $\ggg$ vanishing on $\mmm$, i.e., ${\mmm}^\perp:=\{f\in\ggg^*\mid f(\mmm)=0\}$.
Then $\text{gr}(I_\chi)$ is
the ideal of all functions in $\mathds{k}[\ggg^*]$ vanishing on the closed subvariety $\chi+{\mmm}^\perp$ of $\ggg^*$.
In this way, we  have {the identification}
$\text{gr}(Q_{\chi})\cong\mathds{k}[\chi+\mathfrak{m}^\perp]$, and then $S(\tilde{\mathfrak{p}})\cong\mathds{k}[\chi+\mathfrak{m}^\perp]$.}
\subsubsection{}
 %
As we identify $S(\tilde{\mathfrak{p}})$ with $\mathds{k}[\chi+\mathfrak{m}^\perp]$, we similarly regard the $\mathcal{M}$-algebra $Z_p(\tilde{\mathfrak{p}})$ as the coordinate algebra of the Frobenius twist $(\chi+\mathfrak{m}^\perp)^{(1)}\subseteq(\ggg^*)^{(1)}$ of $\chi+\mathfrak{m}^\perp$, where the natural action of $\mathcal{M}$ on $(\chi+\mathfrak{m}^\perp)^{(1)}$ is a Frobenius twist of the coadjoint action of $\mathcal{M}$ on $\chi+\mathfrak{m}^\perp$.

   \subsubsection{Further description of $\cz_0(\cw)$ via { good transverse} slices} \label{further desc of Z0} Define the Killing isomorphism   $\kappa$ from $\ggg$ to $\ggg^*$  via sending $x\in\ggg$ to $(x,\cdot)$.
       As $\mathds{k}[\chi+\mathfrak{m}^\perp]^{\mathcal{M}}\cong\mathds{k}[\chi+\kappa(\v)]$ by \cite[Lemma 5.1]{GT} or \cite[Lemma 3.2]{Pre3}, write
       \begin{align*}
       Z_p(\tilde{\mathfrak{p}})^{\mathcal{M}}:
=\mathds{k}[(\chi+\mathfrak{m}^\perp)^{(1)}]^{\mathcal{M}}
=\mathds{k}[(\chi+\kappa(\v))^{(1)}],
 \end{align*}
  which is, under the Killing isomorphism, the function algebra on the Frobenius twist of $\kappa(\cs)$ for
  \begin{align}\label{eq: S}
\cs:=e+\v.
\end{align}
{ Note that $\cs$ and $\v$ are isomorphic as affine varieties.  We can regard $\cs$ as a good transverse slice, instead of $\v$,  for the description of $\cz_0(\cw)$.  }

So there is an isomorphism $\phi$ of affine schemes between $\text{Specm}(Z_0(\cw))$ and  the variety ${\cs}$ arising from $\ggg^{(1)*}$ (this $\phi$ is in the same sense as in \S\ref{sec: Frob twist}). More precisely,
\begin{align*}
\text{Specm}(Z_0(\cw))=\tilde\ppp^{(1)*}\slash \cm= \kappa(\cs)^{(1)}.
\end{align*}


 \subsection{Invariants and quotient maps associated with { good transverse} slices}\label{sec: invariant gen}

\subsubsection{}\label{sec: regular ele}
 Let us first recall some basic material on  $G$-invariants of $S(\ggg^*)$ and regular elements of $\ggg$ (see for example,  \cite[\S7.13]{Jan3} or \cite[Ch.4]{Var}).

Recall that an element $X\in\ggg$ is said to be regular if the adjoint orbit $\Ad(G).X$ has the maximal dimension among the adjoint orbits,  which is equal to $2N$, where $N=\#\Phi^+$ denotes the number of positive roots of $G$. Equivalently, the dimension of $\ggg_X$ 
is the smallest one among all centralizers, which is equal to $r$. Denote by $\ggg_\reg$ the set of regular elements of $\ggg$.

Denote by
$$P_X(t) = \det(t.1 -\ad X)$$
the characteristic polynomial of $\ad X$ in the variable $t$. Then we have
\begin{align*}
P_X(t) =\sum_{j=0}^{n} (-1)^{n-j} F_j(X)t^{j},
\end{align*}
where $n=\dim\ggg$ and $F_n\equiv 1$, and all $F_j$ are polynomial function on $\ggg$. All $F_j$ are $\Ad(G)$-invariants. Furthermore, $S(\ggg^*)^G$ are generated by those $F_j$. Note that $\det(\ad X)=0$, and then $F_0\equiv 0$. The smallest integer $t> 0$ such that $F_t\not\equiv0$ is equal to $r$, the rank of $\ggg$. 
Then an element $X\in\ggg$ is regular semisimple if and only if $F_r(X)\ne 0$. The regular semisimple elements constitute an open subset of $\ggg$, which is denoted by $\ggg_{\text{rss}}$.

 \subsubsection{}\label{Res}
 Recall that $r=\rank\,\ggg$, and we  already  assume the hypotheses (H1)-(H3). By Chevalley restriction theorem (see for example, \cite[\S7.12]{Jan3}), there are an isomorphism  $\textsf{Res}: S(\ggg^*)^G \cong S(\hhh^*)^{W}$ as $\bbk$-algebras for any given Cartan subalgebra $\hhh$ of $\ggg$ and the abstract Weyl group $W$ of $\ggg$.
The map $\textsf{Res}$ is actually the restriction of the natural map $\Upsilon: S(\ggg^*)\rightarrow S(\hhh^*)$, sending $f\in S(\ggg^*)$ to $f|_{S(\hhh)}$. Consequently, there is an isomorphism $\textsf{res}$ of varieties:
\begin{align}\label{eq: variHC}
\ggg\slash\hskip-3pt\slash G{\overset{\cong}{\longrightarrow}}\hhh\slash W.
\end{align}
 Then we have the morphism $\zeta:\ggg\rightarrow  \hhh\slash W$. Set $\varpi_\cs:=\zeta|_\cs$. {{Note that $U(\ggg)^G\cong U(\hhh)^{W}$ which are isomorphic to a polynomial ring (see, for example \cite[Lemma 3.3]{BGo} or  \cite[\S9.6]{Jan2}).
      There exists   a set of algebraically independent generators $\{g_1,\ldots,g_r\}$
      in $U(\ggg)^G$ ($\cong U(\hhh)^{W}$).
     Let $\text{gr}_S(g_i)=\tilde g_i$ be the images of $g_i$ associated to the standard grading of $U(\ggg)$. }}
     %
     %
     %
 %
   We can further choose such $g_i$ with $i=1,\ldots,{r}$ such that   the elements $\kappa(\tilde g_1),\ldots,\kappa(\tilde g_r)$ form a free generating set 
     of $S(\ggg^*)^G$.

Note that under above assumption, we have the adjoint quotient map:
\begin{equation*}\label{Steinberg map}
\nu:\ggg\rightarrow \bba^r
\end{equation*}
which sends $x\in\ggg$ to $(\kappa(\tilde g_1)(x),\ldots,\kappa(\tilde g_r)(x))$. Consider the composite of $\nu$ and the canonical isomorphism $\eta: \bba^r\rightarrow \hhh\slash W$. Then we have $\varpi_\cs=\eta\circ\nu|_\cs$.
In virtue of \cite[\S7.3-7.4]{Slo} and \cite[Theorem 5.4 and Remark 5.4]{Pre1}, we have the following result.
    \begin{lemma}\label{Tran sli}
     Under the hypothesis (H3) in \S\ref{sec: good grading}, we consider the restriction of $\zeta$ to $\cs$, i.e.,
    $$\varpi_\cs: \cs\rightarrow  \hhh\slash {W}.$$
The morphism $\varpi_\cs$ is faithfully flat.
\end{lemma}

Consequently, $\varpi_\cs$ is surjective and all  fibers of $\varpi_\cs$ have the same dimension $\sfd-r$ with $\sfd=\dim \ggg_e$ (see \cite{Pre1}, \cite{Slo} or \cite[Proposition 4.11(i)]{SZ}).

Set   $\cs_\reg:=\cs\cap \ggg_\reg$. Then
\cite[Proposition 4.13]{SZ} can be proved in the same way under the hypotheses (H1)-(H3).

   \begin{prop} (\cite[Propositon 4.13]{SZ})\label{prop: 2.2} \label{Regular pts} Under the hypotheses (H1)-(H3), the complement of $\cs_\reg$ in $\cs$ has codimension at least $2$.
   \end{prop}

\subsubsection{}\label{sec: slodowy trans} Let $\sft$  be the translation
\begin{equation*}
\begin{array}{llll}
\sft:&\v&{\overset{\cong}\rightarrow} &\cs\\ &x&\mapsto&e+x,
\end{array}
\end{equation*}
which is an isomorphism of affine varieties. Then one obtains a morphism
\begin{align*}
\psi:=\varpi_\cs\circ \textsf{t}:\v\rightarrow\bba^r,\qquad x\mapsto(\psi_1(x),\ldots,\psi_r(x)).
\end{align*}
The morphism $\psi$ is still faithfully flat with $\psi^{-1}(\xi)\cong\varpi_\cs^{-1}(\xi)$ for any $\xi\in\bba^r$, and the fiber $\psi^{-1}(\xi)$ is normal. We choose a basis $\{x_1, \ldots, {x_\sfd}\}$ of $\ggg_e$. Since $\v$ is dual to $\ggg_e$ via $(\cdot,\cdot)$ (see \S\ref{sec: unip M}), we
let $x^*_1, \ldots, {x_\sfd^*}$ denote the linear functions on $\bbk$
defined via $x^*_i (y)=(x_i , y)$ for $y\in\v$.

The following lemma {can be obtained from} \cite[Corollary 3.4]{BGo} and its proof.
\begin{lemma}\label{lemma: 2.3} Under the hypotheses  (H1)-(H3), the Jacobian  matrix
\begin{equation*}\left(\partial \psi_i\over
\partial x^*_j\right)_{\tiny\begin{array}{l}
1\leqslant i\leqslant r\\1\leqslant j\leqslant {\sfd}
\end{array}}
\end{equation*}
 has rank $r$ in $\v$.

\end{lemma}


\begin{proof}
Let $x\in\v$ be such that $e+x\in \ggg_{\reg}$. Then $(d\psi)_x$ is surjective,  under the hypotheses  (H1)-(H3) (see \cite[\S7.14]{Jan3}). The lemma follows.
\end{proof}

\subsection{Centers of finite $W$-algebras}\label{On the centers of}
Keep  the notations  in the previous sections. In particular, under the hypotheses (H1)-(H3), one  has  the finite $W$-algebras in the way of (\ref{equiW}) and  the {  good transverse slices}  $\cs=e+\v$ (already replacing $\v$ for the convenience of arguments).
In the following  we make more arguments on the center of $\cw(\ggg,e)$ for  the next sections where the geometric descriptions of the Zassenhaus varieties are necessary.

\subsubsection{}\label{sec: Slodowy slices and Z_0(T)} Let us first return to the $p$-center  $Z_0(\cw)=Z_p(\tilde{\mathfrak{p}})^{\mathcal{M}} $.
 Recall that $Z_0(\cw)=\cw({\ggg},e)\cap \varphi(Z_p(\ggg))\cong\varphi(Z_p(\tilde{\mathfrak{p}}))^{\calm}  $ by \eqref{eq: def Z0}. By the discussion in \S\ref{further desc of Z0}, we further have
 \begin{align}\label{eq: Spec Z0}
 \Specm(Z_0(\cw))= \kappa(\cs)^{(1)}.
 \end{align}
In the meanwhile,   one can describe $\cs$ via $\v$ (see \S\ref{sec: slodowy trans}). So we have  an isomorphism of algebras
\begin{align*}
Z_0(\cw)\cong \
 \bbk[\kappa(\v)^{(1)}].
 \end{align*}

 \subsubsection{} Recall that $Z_1(\cw)=\varphi(U(\ggg)^G)$. By the same arguments as in \cite{SZ}, under the present condition (H1)-(H3), we still have the following property for finite $W$-algebras as in \cite[Lemma 3.2]{SZ}.
\begin{lemma}\label{Har Chand Cent}
There exists an isomorphism between $\mathds{k}$-algebras
$$\varphi|_{U(\ggg)^{G}}: U(\ggg)^G\xrightarrow\sim Z_1(\cw),$$
where $\varphi|_{U(\ggg)^{G}}$ denotes the restriction of the map $\varphi$ as in \S\ref{p-center} to the subalgebra $U(\ggg)^{G}$ of $Z(\ggg)$.
\end{lemma}

To simplify notation, in the following discussions we still write $\varphi$ for the map $\varphi|_{U(\ggg)^{G}}$ as in Lemma \ref{Har Chand Cent}.  By Harish-Chandra theorem, we have $S(\hhh)^{W}\cong U(\ggg)^G$. So arising from $\varphi$ there is an isomorphism of varieties
\begin{align}\label{eq: Spec Z1}
\Specm(Z_1(\cw))\cong \hhh^*\slash {W}.
\end{align}

\subsection{The proof of Theorem \ref{central thm}}\label{sec: proof 01}
Denote by $\tilde\cz$ the subalgebra of $Z(\cw)$ generated by $Z_0(\cw)$ and $Z_1(\cw)$. We now explain that the arguments in the proof of \cite[Theorem 0.1]{SZ} are still valid for our present situation with the hypotheses (H1)-(H3).

\subsubsection{Proof for Part (1)}\label{sec: part 1} We still  set $n=\dim\ggg$ and $r=\rank\,\ggg$. By Lemma \ref{Noe Pri}(3), we can consider the fraction field of $Z(\cw)$.
 Set $\bbf_0:=\Frac(Z_0(\cw))$ and $\bbf:=\Frac(Z(\cw))$. Then  $\bbf_0$ is  a subfield of $\bbf$.   Denote by $\cq$  the fractional field of $\tilde\cz$. By the same arguments as in (Step 1) of \cite[\S5.2]{SZ}, we can show the following result.

\begin{lemma}\label{lem: fractional fields} The fraction fields $\cq$ and $\bbf$ coincide.
\end{lemma}

Recall that in \S\ref{221} we have introduced the Kazhdan filtration on $U(\ggg)$ and its sub-quotients, and also their corresponding graded algebras.  
Write $\gr(\tilde{\cz})$ and $\gr(Z(\cw))$ for the graded algebras of $\tilde{\cz}$ and $Z(\cw)$ respectively.   Note that  $\tilde{\cz}\subseteq Z(\cw)$ and $\bbk$ is contained in $\tilde{\cz}$, by inductive arguments on the degrees we can finally deduce the claim
$$\tilde{\cz}=Z(\cw)$$ amounts to  the claim that
\begin{equation}\label{griso}
\gr(\tilde{\cz})=\gr(Z(\cw)).
\end{equation}
From Lemma \ref{lem: fractional fields}  one easily derives that their graded algebras $\gr(Z(\cw))$ and $\gr(\tilde{\cz})$ also share the same fraction field. It follows from Proposition \ref{PI and Central S}(3) that $\gr(Z(\cw))$ is integral over $\gr(Z_0(\cw))$, so, {\it{a fortiori}}, $\gr(Z(\cw))$ is integral over $\gr(\tilde{\cz})$.  Therefore, to prove \eqref{griso} it is sufficient to show that $\gr(\tilde{\cz})$ is integrally closed.

Due to Proposition \ref{prop: 2.2}, by the same arguments as in the (Step 3) of \cite[\S5.2]{SZ} we can show that $\Specm(\gr(\tilde\cz))$ is smooth outside of a subvariety of codimension $2$. On the other hand, due to Lemma \ref{lemma: 2.3} we can prove the following result by the same arguments as in  \cite[Lemma 5.1]{SZ}.
\begin{lemma}\label{complete int} Under the hypotheses (H1)-(H3),
$\Specm(\gr(\tilde\cz))$ is a (strict) complete intersection.
\end{lemma}
 According to Serre's theorem \cite[Chapter 4]{Se}, a complete intersection variety is normal if it is smooth outside of a subvariety of codimension 2. Hence by the above analysis,  $\Specm(\gr(\tilde\cz))$ is normal, i.e.,  $\gr(\tilde\cz)$ is integral closed. This is desired for the proof of the equality $\tilde\cz=Z(\cw)$.  As to the second statement, it follows from the arguments in \cite[Lemma 3.6]{SZ} which works in our current situation.   The proof of (1) is completed.

\subsubsection{The proof for Parts (2) and (3)} The proof of Part (2) is based on the arguments {in} \S\ref{sec: part 1} and the classical result that $Z_1(\ggg)$ is a complete intersection over $Z_0(\ggg)\cap Z_1(\ggg)$ (see for example,  \cite[Theorem 3.5(3)]{BGo}). In the same way as in \cite[\S5.3]{SZ}, we can show that $Z_1(\cw)$ is a complete intersection over $Z_0(\cw)\cap Z_1(\cw)$.  Furthermore, the multiplication map $\mu:~Z_0(\cw)\otimes_{Z_0(\cw)\cap Z_1(\cw)} Z_1(\cw)\rightarrow Z(\cw)$ is an isomorphism of $\mathds{k}$-algebras. This completes the proof of (2).

As to (3),  based on Proposition \ref{prop: 2.2} the proof can be done in the same arguments as in \cite[\S4.8]{SZ}.

\section{Description of the Zassenhaus varieties via { good transverse} slices}\label{sec: Slodowy slices}

Keep the assumptions and notations as in the previous section. In particular, we assume the hypotheses (H1)-(H3).

 \subsection{The regular semisimple elements of good transverse slices}

Now we continue the discussion on the { good transverse} slices under the  hypotheses (H1)-(H3).

 {
Parallel to the  arguments  in \S\ref{sec: regular ele}, we first introduce some other precise criterion for regular semisimple elements. In fact, $\vartheta\in\ggg^*$ is regular semisimple if and only if under {the} $G$-coadjoint action, there exists a $g\in G$ such that $g.\vartheta\in\kappa(\hhh)$ satisfies that $g.\vartheta(h_\alpha)\ne 0$ for all $\alpha\in \Phi^+$ {with $h_\alpha\in\hhh$ being the coroot of $\alpha$.}  By the definition of the Weyl group $W$, the element $H:=\prod_{\alpha\in \Phi}(h_\alpha^p-h_\alpha)\in S(\hhh)$ is $W$-invariant with $h_{-\alpha}=-h_\alpha$ for $\alpha\in\Phi^+$.
We introduce the following criterion polynomial on $\ggg^*$:
\begin{align}\label{eq: fh}
\fh:=\textsf{Res}^{-1}(H)=\textsf{Res}^{-1}(\prod_{\alpha\in \Phi}(h^p_\alpha-h_\alpha))\in\bbk[\ggg]^G,
\end{align}
where $\textsf{Res}$ has the same {meaning} as in \S\ref{Res}.

Recall that $\ggg_{\rss}$ denotes the set of regular semisimple elements of $\ggg$.} It is an open dense subset of $\ggg$ (see \cite{Cart} or  \cite{Jan3}). { Let $\ggg_{\ss}$ denote the set of semisimple elements of $\ggg$.}
Set $\cs_\rss:=\cs\cap \ggg_\rss$ { and $\cs_{\ss}:=\cs \cap \ggg_\ss$.}
\begin{prop}\label{prop: regular ss open den} $\cs_\rss$ is a non-empty open dense subset of $\cs$.
\end{prop}

\begin{proof}  Consider the following commutative diagram of morphisms of varieties
\begin{align}\label{diag: S}
\begin{array}{c}
\xymatrix{G\times\cs\ar[r]^{\gamma}\ar[dr]_{p_2} &\ggg\ar[r]^{\zeta}&\hhh\slash W\\
&\cs\ar[ur]_{\varpi_\cs} &}
\end{array}
\end{align}
where $\gamma$ is defined via mapping $(g,s)$ to $\Ad(g)s$ for $g\in G$ and $s\in\cs$ which just partially explains the meaning of $\cs$ as a good transverse slice to the nilpotent orbit $\Ad(G)e$; $p_2$ stands for the projection on the second component;  the map
$$\zeta: \ggg\rightarrow \ggg\slash\hskip-3pt\slash G{\overset{\textsf{res}}{\longrightarrow}}\hhh\slash W$$
is the adjoint quotient map as defined  in {\S\ref{Res}}; as well as
$\varpi_\cs$ is defined there.

By the same arguments as in \cite[Lemma 5.2]{Slo} along with its proof and remarks  (or see \cite[Theorem 5.4]{Pre1}), we have that the morphism $\varpi_\cs$ is flat, its fibers are normal, and $x\in\cs$ is a regular point of the fiber $\varpi_\cs^{-1}\zeta(x)$ if and only if $x$ is a regular element of $\ggg$. As the arguments in {\S\ref{Res}}, $\varpi_\cs$ is surjective. Hence, by the definitions of regular semisimple elements and of the restriction mapping in the Chevalley restriction theorem (see {\S\ref{sec: invariant gen}}), for a {regular semisimple} element $x\in\ggg$ we have that $\varpi_\cs^{-1}\zeta(x)$ consists of regular semisimple elements.
So $\cs_\rss$ is nonempty. Note that regular semisimple elements form an open dense subset of $\ggg$. Hence $\cs_\rss$ is naturally an open {dense} subset of the irreducible variety $\cs$.
\end{proof}

\begin{remark}\label{rem: 4.2 for cs}
  Note that $\kappa$ gives rise to a $G$-equivariant isomorphism from $\ggg$ to $\ggg^*$.  We have the following diagram as a counterpart to  (\ref{diag: S}),
\begin{align*}
\begin{array}{c}
\xymatrix{G\times\kappa(\cs)\ar[r]^{\gamma^*}\ar[dr]_{p_2^*} &\ggg^*\ar[r]^{\zeta^*} &\hhh^*\slash W\\
&\kappa(\cs)\ar[ur]_{\varpi^*_\cs} &}
\end{array}
\end{align*}
where $\gamma^*, \zeta^*, p_2^*, \varpi_\cs^*$ are the counterparts to $\gamma, \zeta, p_2, \varpi_\cs$, respectively. For example,  $\gamma^*$ means the mapping which sends $(g,\vartheta)$ to $\Ad^*(g)\vartheta$ for $g\in G$ and $\vartheta\in\kappa(\cs)$ where $\Ad^*$ stands for the coadjoint action of $G$ on $\ggg^*$.
\end{remark}

\subsection{Description of the Zassenhaus varieties}\label{sec: main them 3}
The precise arguments for Theorem \ref{central thm}(2) actually implies the following lemma. The {complete} proof for it can be provided as done in \cite[\S5.3]{SZ}.

\begin{lemma}\label{lem: intersec of two centers} The intersection  $Z_0({\cw})\cap Z_1({\cw})$ is a polynomial
algebra of rank $r$, and $Z_1({\cw})$ is a free $Z_0({\cw})\cap Z_1({\cw})$-module with basis $\{f_1^{t_1}\cdots f_r^{t_r}\mid0\leqslant t_i\leqslant p-1, 1\leqslant i\leqslant r\}$.
\end{lemma}


Denote by $\mathcal{V}$ the affine variety defined by $Z_0({\cw})\cap Z_1({\cw})$.
According to Lemma \ref{lem: intersec of two centers} along with (\ref{eq: Spec Z0}) and (\ref{eq: Spec Z1}),  we have the following {projection morphisms} $\pi_1$ and $\pi_2$ of varieties
\begin{align}\label{diag: intersection}
\begin{array}{c}
\xymatrix{\kappa({\cs})^{(1)}
\ar[dr]_{\pi_1}   &\ggg^*\slash\hskip-3pt\slash G\ar[r]^{\textsf{res}^*}
&\hhh^*\slash W  \ar[dl]^{\pi_2} \\
&\mathcal{V} &}
\end{array}
\end{align}
{where $\textsf{res}^*$ denotes the counterpart to $\textsf{res}$ as in \eqref{eq: variHC}}.

Still denote by $\scrz$  the Zassenhaus variety of $\cw(\ggg,e)$.  As a direct consequence of Theorem \ref{central thm} along with the above lemma, we have

 \begin{theorem}\label{cor: fib prod geometry1} There is a canonical isomorphism of schemes via the diagram (\ref{diag: intersection}):
 \begin{align}\label{eq: fiber product}
 \mathscr{Z}{\overset\cong\longrightarrow} \kappa(\cs)^{(1)}
 \times_{\calv}\hhh^*\slash W \;\; {\overset \cong\longleftarrow} \kappa(\cs)^{(1)}
 \times_{\calv}\ggg^*\slash\hskip-3pt\slash G.
 \end{align}
\end{theorem}

 {This yields Theorem \ref{thm: main 2}.}

\begin{remark}\label{rem: JC decomp}
 Recall $\mathfrak{F}: \hhh^*\rightarrow \hhh^{*(1)}$ is defined via $\mathfrak{F}(\lambda)(h)= \lambda(h)^p-\lambda(h^{[p]})$ for $\lambda\in\hhh^*$ and $h\in\hhh$.
From Theorem \ref{cor: fib prod geometry1}, it is  yielded that $\scrz$ can be identified with a closed subvariety of the affine space
 $$\kappa(\cs)^{(1)}\times \hhh^*\slash W$$
 whose elements $(\vartheta, \zeta^*(\lambda))$ with  $\zeta^*:\hhh^*\rightarrow \hhh^*\slash W$ {being} the same as  in Remark \ref{rem: 4.2 for cs}  are  subject to the constraint:
 the semisimple part $\vartheta_s$ in the Jordan-Chevalley decomposition $\vartheta=\vartheta_s+\vartheta_n$ is $G$-conjugate to some element
  $\vartheta'_s\in \hhh^*$ satisfying that
 \begin{align}\label{eq: conj ss}
 \vartheta'_s(h)=\mathfrak{F}(\lambda)(h)= \lambda(h)^p-\lambda(h^{[p]})
 \end{align}
  for $h\in\hhh$. Here, it's worth reminding the following facts.

    (1) If the semisimple part $\vartheta_s$ is conjugate  to another element $\vartheta_s''\in \hhh^*$,  which means, there are $g',g''\in G$ such that $\vartheta_s'=\Ad^*(g')\vartheta_s$ and $\vartheta_s''=\Ad^*(g'')\vartheta_s$, {we claim that
    $\vartheta_s''=w(\vartheta_s')$ for some $w\in W$.  To achieve this, we just need to show
    \begin{align}\label{eq: Jan}
     \Ad^*(G)(\vartheta'_s)\cap\mathfrak{h}^*=\{w(\vartheta'_s)\mid w\in W\},
     \end{align}
     where $\Ad^*(G)$ denotes the coadjoint action of $G$ on $\g^*$.
This can be proved by the same consideration as in \cite[\S 7.12]{Jan3}. For the readers' convenience, we give detailed  arguments. Recall that the notation $G_\lambda$ denotes the stabilizer of $\lambda\in \ggg^*$ in $G$. This means $G_\lambda=\{g\in G\mid \Ad^*(g)\lambda=\lambda\}$.
The inclusion ``$\supset$"  in (\ref{eq: Jan}) is clear.  For the reverse inclusion, we consider any given $g^{-1}\in G$ with $\Ad^*(g^{-1})(\vartheta'_s)\in\mathfrak{h}^*$.
  Note that $T$ is contained in $G_{\vartheta'_s}$ and $G_{\Ad^*(g)(\vartheta'_s)}=g^{-1}(G_{\vartheta'_s})g$, hence a maximal torus in these two groups. So $T$ and $gTg^{-1}$ are maximal tori in $G_{\vartheta'_s}$ and there exists $g_1\in G_{\vartheta'_s}$
with $g_1Tg_1^{-1}=gTg^{-1}$; it follows that $g^{-1}g_1$ normalises $T$. Note that $w(\vartheta'_s$) with $w\in W$ are just the $\Ad^*(n)(\vartheta'_s)$ with $n$
 being in the normaliser of $T$. Hence  $\Ad^*(g^{-1})(\vartheta'_s)=\Ad^*(g^{-1}g_1)(\vartheta'_s)$ is equal to some $w(\vartheta'_s)$ with $w\in W$. This
yields (\ref{eq: Jan}).  So the constraint (\ref{eq: conj ss}) is well-defined for $\kappa(\cs)^{(1)}\times \hhh^*\slash W$.}

        (2) Recall that there is a polynomial $f(x)\in \bbk[x]$ such that $\vartheta_s=f(\vartheta)$ (see for example,  \cite[\S1.6.3]{GW}).  So  $\vartheta_s'=\Ad^*(g')\vartheta_s$ with $g'\in G$ is a regular function on $\vartheta$ (see for example,  \cite[\S1.4.2]{GW}).

\end{remark}

{

\subsection{}\label{sec: key lemma} We will introduce a key lemma on the semisimple elements of  $\kappa(\cs)$, which will be used in \S\ref{sec: 4}.  Fix a Chevalley basis $\BbbB$ in $\ggg$ which consists of the root vectors $X_\alpha$ for $\alpha$ ranging over $\Phi$ and the toral basis $h_i$ with $i=1,\ldots,r$ of $\hhh$  such that $h_j=h_{\alpha_j}$ for the simple root system $\Delta=\{\alpha_j\mid j=1,\ldots,\text{rank}([\ggg,\ggg])\}$ of $\Phi^+$. Then $\BbbB$ spans a restricted Lie algebra $\ggg_{\bbf_p}$ over $\bbf_p$,  and  $\ggg_{\bbf_p}$ is an $\bbf_p$-form of $\ggg$.

\begin{lemma}\label{lem: Fro closed} Keep the notations as above. Let $\vartheta\in \kappa(\cs)$ be a semisimple element which is conjugate to $\vartheta'\in \hhh^*$ by the action of $\Ad^*(g^{-1})$ for $g\in G$.  Then  $\Ad^*(g)(\mathfrak{F}(\vartheta'))\in \kappa(\cs)_{ss}^{(1)}$.  Conversely, for any $\lambda\in \hhh^*$ with  $\Ad^*(g)(\mathfrak{F}(\lambda))\in \kappa(\cs)^{(1)}$, then $\Ad^*(g)\lambda\in \kappa(\cs)_{ss}$.
All the above also hold in the regular semisimple situation.
\end{lemma}


\begin{proof}
Diagrammatically, we will show   that the following diagram is commutative
\begin{itemize}
\item[(\textbf{diag})]:
\[
\begin{CD}
\vartheta \in\kappa(\cs)_{ss} @>{\Ad^*(g^{-1})}>>\vartheta'=\Ad^*(g^{-1})\vartheta \;\in \hhh^*\\
@V{\textsf{AS}}|_{\kappa(\cs)}VV @V{\mathfrak{F}}VV\\
 \Ad^*(g)(\mathfrak{F}(\vartheta'))\in\kappa(\cs)_{ss}^{(1)}@<<{\Ad^*(g)}<  \mathfrak{F}(\vartheta')\in \hhh^{*(1)}
\end{CD},
\]
\end{itemize}
where $\textsf{AS}|_{\kappa(\cs)}$  on the left vertical arrow denotes the restriction of $\textsf{AS}$ to $\kappa(\cs)$, the latter of which
means the map from $\ggg^*$ to $\ggg^{*(1)}$  such that  for any $\vartheta\in \ggg^*$, $\textsf{AS}(\vartheta)(x)=\vartheta(x)^p-\vartheta(x^{[p]})$ for all $x\in\ggg$.





First note that by definition,  $\vartheta\in \kappa(\cs)_{ss}$ is equivalent to say  $\vartheta=(e+v,\cdot)$ for $e+v\in\cs_{ss}$. Furthermore, $\vartheta'=\Ad^*(g^{-1})(\vartheta)\in \hhh^*$ implies that $\vartheta'=(\Ad(g^{-1})(e+v),\cdot)\in\hhh^*\subset \ggg^*$. Equivalently, $\Ad(g^{-1})(e+v)\in \hhh$.
By definition, $\mathfrak{F}(\vartheta')$ is defined via $\mathfrak{F}(\vartheta')(x)=\Ad^*(g^{-1})(\vartheta)(x)^p-\Ad^*(g^{-1})(\vartheta)(x^{[p]})$ for any $x\in \ggg$. This means that $\mathfrak{F}(\vartheta')(x)=(\Ad(g^{-1})(e+v),x)^p-(\Ad(g^{-1})(e+v),x^{[p]})$.  Note that $g\in G$ is commutative with the $p$-mapping $[p]$, and preserves the form $(\cdot,\cdot)$ invariant, which also gives rise to an automorphism of $U(\ggg)$.
We have
\begin{align*}
\textsf{AS}|_{\kappa(\cs)}(\vartheta)(x)&= (e+v, x)^p-(e+v, x^{[p]})\cr
&=(\Ad(g^{-1})(e+v),\Ad(g^{-1})(x))^p-(\Ad(g^{-1})(e+v), (\Ad(g^{-1})(x))^{[p]})\cr
&=\Ad^*(g)(\mathfrak{F}(\vartheta'))(x).
\end{align*}
Thus, $\textsf{AS}|_{\kappa(\cs)}(\vartheta)=\Ad^*(g)(\mathfrak{F}(\vartheta'))$. This proves that the diagram (\textbf{diag}) is commutative, and $\Ad^*(g)(\mathfrak{F}(\vartheta'))=\mathfrak{F}(\vartheta')\circ\Ad(g^{-1})$ lies in $\kappa(\cs)^{(1)}$.

Conversely,
suppose that $\lambda\in\hhh^*$ satisfies  $\theta:=\Ad^*(g)(\mathfrak{F}(\lambda))\in \kappa(\cs)^{(1)}$ with $g\in G$.  Set  $\vartheta:=\Ad^*(g)(\lambda)$. We shall prove $\vartheta\in \kappa(\cs)$. By the non-degeneracy of $(\cdot,\cdot)$ on $\ggg$, we can write $\vartheta=(w,\cdot)$. We want to show $w\in\cs$.

%
Recall that $\kappa(\cs)=
\{(e+v,\cdot):\ggg\rightarrow \bbk, x\mapsto (e+v,x),\;\forall x\in\ggg\mid v\in \v\}$. One  can describe  $\kappa(\cs)^{(1)}$ via  $\kappa(\cs)$  with a Frobenius twist as below. By the definition of Frobenius twist (see for example,  \cite[\S{I}.9.1]{Jan1}),  in the above setup $\kappa(\cs)^{(1)}$ consists of all functions $((e+v,\cdot)):\ggg\rightarrow \bbk, x\mapsto ((e+v,x))$ for $v\in\v$  such that $((e+v,x_0))=(e+v,x_0)$ for $x_0\in \BbbB$, and $((e+v, ax_0+bx_0'))= a^{p^{-1}}(e+v,x_0)+b^{p^{-1}}(e+v,x'_0)$ for any $a,b\in \bbk$ and $x_0,x_0'\in \BbbB$.  Hence,  for $\theta\in\kappa(\cs)^{(1)}$ we have $e+v\in \cs$ with $v\in \v$ such that $\theta=((e+v,\cdot))$.
Then we have
\begin{align}\label{eq: linear fun S exp}
\theta(x)=((e+v,x)) \text{ for }x\in\ggg.
\end{align}

On the other hand, similarly to the arguments on (\textbf{diag}) we can obtain the following commutative diagram
\[
\begin{CD}
\Ad^*(g)(\lambda)=\vartheta\in\ggg_{ss}^* @<{\Ad^*(g)}<<\lambda\in \hhh^*\\
@V{\textsf{AS}}VV @V{\mathfrak{F}}VV\\
 \Ad^*(g)(\mathfrak{F}(\lambda))\in\ggg_{ss}^{*(1)}@<<{\Ad^*(g)}<  \mathfrak{F}(\lambda)\in \hhh^{*(1)}
\end{CD}.
\]
Hence we have $\textsf{AS}(\vartheta)=\theta$.

Keep it in mind that $\theta\in \kappa(\cs)^{(1)}$ and note that both $\textsf{AS}|_{\kappa(\cs)}:\kappa(\cs)\rightarrow \kappa(\cs)^{(1)}$ and $\textsf{AS}|_{\kappa(\cs)_{ss}}:\kappa(\cs)_{ss}\rightarrow \kappa(\cs)_{ss}^{(1)}$  are epimorphisms.
Comparing with (\ref{eq: linear fun S exp}) and applying the non-degeneracy of $(\cdot,\cdot)$ on $\ggg$ again, we have $w\in \cs$.  Hence $\vartheta$ belongs to $\kappa(\cs)$, which is desired.

As to the last statement, it follows from the similar arguments.
\end{proof}

}

\section{Rationality of the Zassenhaus varieties}\label{sec: 4}
Recall that  an affine  variety/scheme $X$ is called rational if the {{fraction  field}} of $\bbk[X]$ is purely transcendental.
Lemma \ref{lem: Fro closed} enables us to proceed with the arguments in the same spirit as  Tange did in \cite{T}. All assumptions and notations in the previous sections  are maintained.

\subsection{The varieties $\co_{\cs}$  and $\cU_\cs$}\label{sec: rss}

Set $$\co_{\cs}:=\{\vartheta=\vartheta_s\in \kappa(\cs)\mid \fh(\vartheta'_s)\ne 0\},$$
where $\fh$ is defined in (\ref{eq: fh}) and the meaning of $\vartheta_s$ and $\vartheta'_s$ are the same as in Remark \ref{rem: JC decomp}. The following property is a consequence of  the definition of regular semisimple elements and the counterpart of Proposition \ref{prop: regular ss open den}.

\begin{lemma}\label{lem: key obser}The following hold:

(1) $\co_{\cs}$ coincides with $\kappa(\cs)_\rss$.

(2) The variety $\co_{\cs}$ is non-empty open and dense in $\kappa(\cs)$.
\end{lemma}

{
{ In the following, we set
 $$ \mathsf{N}=\text{Nor}_G(T),$$ the normalizer of $T$ in $G$.
Consider
$$\cU_{\cs}:=\{(\chi, g)\in \kappa(\cs)_\rss\times G \mid \Ad^*(g^{-1})(\chi)\in \hhh^* \}.$$
There is an action of $\mathsf{N}$ on  $\cU_{\cs}$ defined via for $(\chi, g)\in \cU_\cs$ and $n\in\sfn$,
\begin{align}\label{eq: n action on cs}
n\star(\chi,g)=(\chi, gn^{-1}) \; \text{for }n\in \sfn.
\end{align}
This action is well-defined. Actually, from the fact that the action of $n\in \mathsf{N}$ on $\hhh^*$ amounts to an action of some $w\in W$ on $\hhh^*$, we have
\begin{align}\label{eq: check 0}
\Ad^*(gn^{-1})^{-1}(\chi)=\Ad^*(ng^{-1})(\chi)=\Ad^*(n)(\Ad^*(g^{-1})\chi)\in \hhh^*.
\end{align}
The set $\cU_\cs$ is actually a closed subset of  $\kappa(\cs)_{\rss}\times G$.
By the construction of $\co_\cs$ and $\cU_\cs$, the following lemma is clear.

\begin{lemma}\label{osus}
 There is an isomorphism between $\co_\cs$ and $\cU_\cs\slash \sfn$.
\end{lemma}

\begin{proof} The lemma follows from Remark \ref{rem: JC decomp}(1).
\end{proof}

\subsection{The varieties $\cC_{\cs}$ and {$\scrz^\flat$}}
\subsubsection{}\label{sec: xi} Note that $\hhh$ is a toral subalgebra. 
There is a basis consisting of toral elements which means all of them satisfying $h^{[p]}=h$ (see for example, \cite[Theorem 3.6 of Chapter 2]{SF}).
Thus we can choose a  toral $\bbk$-basis $\{h_1,\ldots, h_r\}$ in $\hhh$ whose $\bbf_p$-linear span is denoted by $\hhh_{\bbf_p}$ (we can take them as the same as we did in \S\ref{sec: key lemma}). Then $\hhh_{\bbf_p}$ is an $\bbf_p$-form of $\hhh$.

By Remark \ref{rem: JC decomp}(1),  we define a variety
\begin{align}\label{eq: cs def}
\cC_{\cs}:=\{&(\vartheta^p,\lambda, g)\in \kappa(\cs)^{(1)}_\rss\times  \hhh^* \times G \cr
 &\mid   (\vartheta')^p(h)=\lambda(h)^p-\lambda(h^{[p]})  \text{  for all } h\in\hhh \cr
 &\qquad \text{ with }\vartheta'=\Ad^*(g^{-1})\vartheta\in \hhh^{*}
\}.
\end{align}
Then $\cC_{\cs}$ is a closed subvariety of $\kappa(\cs)^{(1)}_\rss \times \hhh^*\times G$, which can be described in more detail below. The morphism $\mathfrak{F}$  from $\hhh^*$ to $\hhh^{(1)*}$ sends $\lambda$ to $\mathfrak{F}(\lambda)$ which takes value $\lambda(h)^p-\lambda(h^{[p]})$ at $h\in \hhh$.
Then $\mathfrak{F}(\lambda)=0$ for all $\lambda\in (\hhh_{\bbf_p})^*$.
So $\cC_{\cs}$ is just the solution space of the equation $(\vartheta')^p=\mathfrak{F}(\lambda)$ in the variety  $\kappa(\cs)_\rss \times \hhh^*\times G$.   More precisely,  write $\vartheta_i$ for the functional $h_i\mapsto \vartheta'(h_i)$, and $\lambda_i$ for the functional $h_i\mapsto \lambda(h_i)$. Thus, $\cC_{\cs}$ can be regarded  a closed subset of $\kappa(\cs)^{(1)}_\rss \times \hhh^*\times G$ defined by the equations
\begin{align}
&\vartheta_i^p=\lambda_i^p-\lambda_i,  \;\; i=1,\ldots,r,  \cr
&\vartheta'=\Ad^*(g^{-1})\vartheta\in \hhh^{*}.
\end{align}

\subsubsection{} Note that the action of $n\in \mathsf{N}$ on $\hhh^*$ amounts to an action of some $w\in W$ on $\hhh^*$. So we can define the action of $\sfn$ on $\cC_\cs$ as in the following lemma.
\begin{lemma}
There is an action of $\sfn$ on $\cC_\cs$ defined via for $(\vartheta^p,\lambda, g)\in \cC_\cs$ and $n\in\sfn$,
\begin{align}\label{eq: n action on cc}
n\maltese (\vartheta^p,\lambda, g)=(\vartheta^p,\Ad^*(n)(\lambda), gn^{-1}).
\end{align}
\end{lemma}

\begin{proof}
It suffices to show that for $(\vartheta^p, \lambda, g)\in \cC_\cs$ with
 $$(\vartheta')^p(h)=\lambda(h)^p-\lambda(h^{[p]})  \text{  for all } h\in\hhh \text{ with }\vartheta'=\Ad^*(g^{-1})\vartheta\in \hhh^{*}, $$
 then
 \begin{align}\label{eq: check 1}
 \vartheta''=\Ad^*((gn^{-1})^{-1})\vartheta\in \hhh^*,
 \end{align}
  and furthermore, for all $h\in \hhh$ the following holds
  \begin{align}\label{eq: check 2}
(\vartheta'')^p(h)=((\Ad^*(n)(\lambda))h)^p-(\Ad^*(n)\lambda)(h^{[p]})  \text{  for all } h\in\hhh.
 \end{align}
The verification on (\ref{eq: check 1}) can be done as the same as  in (\ref{eq: check 0}). In particular, the action of $n\in \mathsf{N}$ on $\hhh^*$ amounts to an action of some $w\in W$ on $\hhh^*$. So we have
\begin{align}\label{eq:  vartheta prime and prime}
\vartheta''=w(\vartheta')\in \hhh^*.
\end{align}

As to (\ref{eq: check 2}), we first note that  the action of $\Ad^*(n)$ on $\hhh^*$ gives rise to an automorphism of the symmetric algebra  $\text{Sym}(\hhh^*)$, and also the action preserves the $p$-mapping $[p]$. Then from (\ref{eq:  vartheta prime and prime}),  we have
\begin{align}
(\vartheta'')^p(h)&=(\Ad^*(n)(\vartheta'))^p(h)\cr
&=\vartheta'(\Ad(n^{-1})h)^p\cr
&=((\Ad^*(n)\lambda)h)^p-(\Ad^*(n)\lambda)(h^{[p]})
 \end{align}
 for all  $h\in\hhh$.
This completes the proof.
\end{proof}

\begin{corollary}\label{cor: the last cor} There is an isomorphism of varieties from $\cC_\cs\slash \sfn$ to the variety $\kappa(\cs)_\rss^{(1)}\times_{\mathcal{V}}\hhh\slash W$.
\end{corollary}

\begin{proof} Note that the action of $\sfn$ on $\hhh^*$ amounts to the action of the Weyl group $W$ on $\hhh^*$. Conversely, for any $w\in W= \sfn\slash T$, the action of $w$ on $\hhh^*$ can lift to  the action of $\dot w\in \sfn$ (in the following, we still denote  $\dot w$ by $n$ for the uniformity  of notations in the context).

Consider the following variety
\begin{align*}
\digamma:=&\{(\vartheta^p,\zeta^*(\lambda))\in \kappa(\cs)^{(1)}_\rss\times  \hhh^*\slash W  \mid   (\vartheta')^p(h)=\lambda(h)^p-\lambda(h^{[p]})\;\; \forall h\in\hhh \cr
  &\hskip6cm
\text{ with }\vartheta'=\Ad^*(g^{-1})\vartheta\in \hhh^{*} \text{ for some }g\in G\},
\end{align*}
for which it is worthwhile reminding that for any $w\in W$ and the corresponding $n\in \sfn$, $\vartheta'':=\Ad^*((gn^{-1})^{-1})\vartheta=\Ad^*(n)\vartheta'\in \hhh^*$, we have
\begin{align}\label{eq: another eq}
\vartheta''^{p}(h)=&w(\lambda)(h)^p-w(\lambda)(h^{[p]}) \text{ for all }h\in\hhh,\cr
&\qquad \text{  with }\vartheta''=\Ad^*((gn^{-1})^{-1})\vartheta.
\end{align}
This shows that for any $w\in W$, the equations (\ref{eq: another eq}) is  well-defined, which define the point   $(\vartheta^p, \zeta^*(w(\lambda)))$ (the same point as  $(\vartheta^p, \zeta^*(\lambda))$).
Furthermore, the equations in (\ref{eq: another eq}) show that $\digamma$ is an $\sfn$-variety with  $\sfn$-action being the identify effect.
Now the equations appearing in the Definition of (\ref{eq: cs def}) and the equations in (\ref{eq: another eq}) give rise to a natural $\sfn$-equivariant morphism from $\cC_\cs$ to $\digamma$,  sending  the point $(\vartheta^p, \lambda, g)\in \cC_\cs$ to the point $(\vartheta^p, \zeta^*(\lambda))\in \digamma $.  This morphism finally induces the $\sfn$-quotient morphism  $\cC_\cs\slash \sfn$ to $\digamma$,  which is an isomorphism of varieties.
%
%


By definition,  $\kappa(\cs)_\rss^{(1)}\times_{\mathcal{V}}\hhh^*\slash W$ can be identified with $\digamma$. Hence the corollary is proved.
%
\end{proof}

\subsubsection{} Now  we  establish a connection between the different varieties we introduced above.
\begin{prop}\label{prop: eq iso 2}
There exists  an $\sfn$-equivariant isomorphism of $\sfn$-varieties
\begin{align*}
\cC_{\cs} {\overset{\cong}{\longrightarrow}} \cU_{\cs}.
\end{align*}
\end{prop}

\begin{proof}   For $(\vartheta^p,\lambda,g)\in \cC_\cs$, by definition we have $\vartheta':=\Ad^*(g^{-1})\vartheta\in \hhh^{*}$ for $g\in G$, and also
\begin{equation}\label{vartheta'}
(\vartheta')^p(h)=\lambda(h)^p-\lambda(h^{[p]})  \text{  for all } h\in\hhh.
\end{equation}
By Lemma \ref{lem: Fro closed}, $\Ad^*(g)\lambda\in \kappa(\cs)_{\rss}$.
Define a map $\psi_1$ from $\cC_{\cs}$ to $\cU_{\cs}$ via sending
 $(\vartheta^p,\lambda, g)$ to $(\Ad^*(g)\lambda, g)$. By a direct check, it is $\sfn$-equivariant.

 For the inverse of  $\psi_1$, we define a map $\psi_2:\cU_{\cs}\longrightarrow \cC_{\cs}$  by sending any element  $(\chi,g)\in \cU_\cs$ with $\lambda:=\Ad^*(g^{-1})\chi\in \hhh^*$  to $(\Ad^*(g)\mathfrak{F}(\lambda), \lambda,  g)\in\cC_{\cs}$. Again we note that the action of $\Ad^*(n)$ on $\hhh^*$ for $n\in \sfn$ gives rise to an automorphism of the symmetric algebra  $\text{Sym}(\hhh^*)$, and also the action preserves the $p$-mapping $[p]$. We still have that $\psi_2$ is $\sfn$-equivariant.

  Both $\psi_1$ and $\psi_2$ are clearly morphisms of varieties. Furthermore,  it is easily confirmed by a straightforward computation that the compositions between  $\psi_1$ and $\psi_2$ are identities.  Hence, both $\cU_\cs$ and $\cC_\cs$ are   $\sfn$-equivariant isomorphic,  as $\sfn$-varieties.
   This completes the proof.
 \end{proof}
}}

\subsubsection{}
From now on, we identify $\scrz$ with  $\kappa(\cs)^{(1)}\times_{\calv}\hhh^*\slash W$. Consider the projection on the first component for the decomposition of $\scrz$ in (\ref{eq: fiber product}):
$$\pr1: {\scrz}\rightarrow \kappa(\cs)^{(1)}.$$
The homomorphism $\pr1$ of varieties has the complete inverse image
$${\scrz^\flat}:=\pr1^{-1}(\kappa(\cs)^{(1)}_\rss).$$

Then we have the following result.
\begin{prop}\label{prop: eq iso 1}The following is true:
\begin{align*}
\cC_{\cs}\slash \sfn= {\scrz^\flat}.
\end{align*}
\end{prop}
\begin{proof}  It is a direct consequence of Corollary \ref{cor: the last cor}.
\end{proof}

\subsection{The main result}\label{sec: proof 02} We are in a position to introduce the main result of this section.

\begin{theorem}\label{thm: main}  The variety ${\scrz^\flat}$ is isomorphic to $\kappa(\cs)_\rss$. Consequently,  $\scrz$ is rational.
\end{theorem}

\begin{proof}
Summing up  Lemmas \ref{lem: key obser} and \ref{osus}   along with Propositions \ref{prop: eq iso 2} and \ref{prop: eq iso 1}, the first part of the theorem follows. As to the second part, {it follows from the definition of ${\scrz^\flat}$ and Proposition \ref{prop: regular ss open den} that}  ${\scrz^\flat}$ is a {non-empty} open dense subset of the irreducible variety $\scrz$. So $\scrz$ is bi-rationally equivalent to the affine space $\kappa(\cs)$. The proof is completed.
\end{proof}

 This yields Theorem \ref{thm: rationality}.

\subsection*{Acknowledgement} The authors are deeply indebted to  the anonymous referee for helpful comments and valuable suggestions which enable them to make corrections and improvements on the earlier versions of the manuscript.


\begin{thebibliography}{A1}
\bibitem{A} T. Arakawa, {\em Introduction to W-algebras and their representation theory}, in: Perspectives in Lie Theory, in: Springer INdAM Series, vol. 19, 2017, 179-250.





\bibitem{BGl} K. A. Brown and K. R. Goodearl, {\em Homological aspects of Noetherian PI Hopf algebras and irreducible modules of maximal dimension}, J. Algebra 198 (1997), 240-265.

\bibitem{BGo}  K. A. Brown and I. Gordon, {\em The ramification of centres: Lie algebras in positive characteristic and quantised enveloping algebras}, Math. Z. 238 (2001), 733-779.


\bibitem{Cart} R. W. Carter, {\em Finite groups of Lie type: conjugacy classes and complex characters}, John Wiley \& Sons, 1985.








\bibitem{GW} R. Goodman and N. R. Wallach, {\em Symmetry, representations, and invariants}, GTM 255,  Springer, Dordrecht, 2009.

\bibitem{GT} S. M. Goodwin and L. W. Topley, {\em Modular finite $W$-algebras},  Internat. Math. Res. Notices 18
(2019), 5811-5853.





\bibitem{Jan1} J. C. Jantzen, {\em Representations of algebraic groups}, 2nd ed., Math. Surveys Monogr. 107, American Mathematical Society, Providence,
2003.


\bibitem{Jan2} J. C. Jantzen, {\em Representations of Lie algebras in prime characteristic}, in: Representation Theories and Algebraic Geometry, Proc. Montr\'{e}al (NATO ASI series C 524) (1997), 185-235.

\bibitem{Jan3} J. C. Jantzen, {\em Nilpotent orbits in representation theory}, Progress in Math., vol. 228, Birkh\"{a}user, 2004.



\bibitem{KW}  V. Kac and B. Weisfeiler, {\em Coadjoint action of a semisimple algebraic group and the center of the enveloping algebra in characteristic $p$}, Indag. Math. 38 (1976), 136-151.



\bibitem{Ko} B. Kostant, {\em On Whittaker vectors and representation theory}, Invent. Math. 48 (1978), 101-184.



\bibitem{L4} I. Losev, {\em Finite W-algebras}, in: Proceedings of the International Congress of Mathematicians, Vol. III, Hindustan Book Agency, New Delhi, 2010, 1281-1307.

\bibitem{Lyn} T. E. Lynch, {\em Generalized Whittaker vectors and representation theory}, PhD thesis, M. I. T., 1979.



\bibitem{MiRu} I. Mirkovi\'{c} and D. Rumynin,  {\em Centers of reduced enveloping algebras}, Math. Z. 231 (1999), 123-132.




    \bibitem{Po}  K. Pommerening, {\em {$\ddot{U}ber$} die unipotenten Klassen reduktiver Gruppen II}, J. Algebra 65 (1980), 373-398.


\bibitem{Pre1} A. Premet, {\em Special transverse slices and their enveloping algebras}, Advances in Math. 170 (2002), 1-55.



\bibitem{Pre3} A. Premet, {\em Commutative quotients of finite $W$-algebras}, Advances in Math. 225 (2010), 269-306.

\bibitem{Pre2} A. Premet, {\em  Enveloping algebras of Slodowy slices and Goldie rank}, Transform. Groups 16 (2011), 857-888.


\bibitem{PT} A. Premet, R. Tange, {\em Zassenhaus varieties of general linear Lie algebras}, J. Algebra 294  (2005), 177-195.


\bibitem{SZ} B. Shu and Y. Zeng, {\em  Centers and Azumaya loci for finite  $W$-algebras in positive characteristic},
Math. Proc. Cambridge Philos. Soc. 173 (2022), 35-66.


\bibitem{Slo} P. Slodowy, {\em Simple singularities and simple algebraic groups}, Lecture Notes in Mathematics 815, Springer-Verlag, Berlin/Heidelberg/New York, 1980.

\bibitem{Spa} N. Spaltenstein, {\em Existence of good transversal slices to nilpotent orbits in good characteristics}, J.  Fac.  Sci.  Univ.  Tokyo (IA) 31 (1984), 283-286.


\bibitem{Se} J.-P. Serre, {\em Algebre locale---Multiplicites}, Lecture Notes in Math., vol. 11, Springer-Verlag, Berlin-New York, 1975.


\bibitem{ST} D. I. Stewart and A. R. Thomas, {\em The Jacobson-Morozov theorem and complete reducibility of Lie subalgebras}, Proc. Lond. Math. Soc. (3) 116 (2018), no. 1, 68–100.


\bibitem{SF} H. Strade and R. Farnsteiner, {\em Modular Lie algebras and their representations}, Monographs and textbooks in pure and applied mathematics; V. 116, Marcel Dekker, Inc., New York, Basel, 1988.

\bibitem{T}  R. Tange,  {\em The Zassenhaus variety of a reductive Lie algebra in positve characteristic}, Advances in Math. 224 (2010), 340-354.

\bibitem{Ve} F. D. Veldkamp, {\em The center of  the universal enveloping algebra of a Lie algebra in characteristic $p$}, Ann. Sci. Ecole Norm. Sup. 5 (1972), 217-240.

\bibitem{Var} V. S. Varadarajan, {\em Lie groups, Lie algebras, and their representations}, GTM 102, Springer-Verlag, New York, 1984.

\bibitem{W} W. Wang, {\em Nilpotent orbits and finite $W$-algebras}, Fields Inst. Commun. 59 (2011), 71-105.

\end{thebibliography}
\end{document}